\documentclass[11pt,a4paper]{article}
\usepackage[utf8]{inputenc}
\usepackage[english]{babel}
\usepackage{amssymb, amsmath, enumerate}
\usepackage{graphicx}
\usepackage{epstopdf}
\usepackage{subcaption}
\usepackage{authblk}
\usepackage{textcomp}
\newcommand*{\No}{\textnumero}

\newtheorem{theorem}{Theorem}[section]

\newenvironment{proof}[1][]{\begin{trivlist}
\item[\hskip \labelsep {\bfseries #1}]}{\end{trivlist}}

\newcommand{\qed}{\nobreak \ifvmode \relax \else
      \ifdim\lastskip<1.5em \hskip-\lastskip
      \hskip1.5em plus0em minus0.5em \fi \nobreak
      \vrule height0.75em width0.5em depth0.25em\fi}

\begin{document}
\title{Game-Theoretical Strategy of Robot in the Area with Dynamical Obstacles}
\author[1]{Oleg Malafeyev\thanks{malafeyevoa@mail.ru}}
\affil[1]{Saint-Petersburg State University,  Russia}
\date{}
\maketitle
\begin{abstract}
The effectiveness of a robot manipulation to a large extent is determined by the speed of making this or that movement
needed for carrying out the task. Accordingly to this the problem of optimal robot control is often subdivided into two
subproblems solved separately. In an autonomous regime the trajectory planning is fulfilled for providing the robot movement
time close to the minimal.
\par
The problem of a robot reaching a moving aim under the existence of movable (or in particular immovable) obstacles is considered in the report. 
This problem is represented as two-person zero-sum game, taking place in a metric space $X$.
\end{abstract}



\section{Introduction}

Let us suppose that at every moment of the conflict process the robot is informed of the trajectory of an obstacle and also a rule of changing the obstacle configuration up to this moment. By $P_1'(x_0,t_0,t)$ we denote the robot’s attainability function in the space $X$, and by $O(t)$-the position of the obstacle in $X$. Thus $O(t)$ is supposed to be a compact subset of $X$.
\par
Let us consider the multivalued function $P_1'(x_0,t_0,t)\cap O(t)=P_1(x_0,t_0,t)$ appearing here. We interpret the set $P_1(x_0,t_0,t)$ as a set of the positions which the robot may reach in under restriction imposed by moving obstacles. By that a quasi-robot can be introduced, with the dynamics described by the attainability function $P_1(x_0,t_0,t)$. It is shown that this function satisfies all the axioms setting a general dynamic system in $X$.
\par
Similar reasoning is valid for the aim $P_2$ for which the attainability function $P_2(x_2,t_0,t)$ is introduced.
\par
Now the problem of the robot and the aim meeting is represented mathematically quite strictly as a dynamic game in the space $X$.

\section{Analysis}
Let us consider a problem of the robot’s pursuit of the moving aim (e.g. another robot) under the existence of moving obstacles. 
This problem may be reduced to a dynamic conflict process.
Such processes take place in a complete locally compact metric space $X$. The maneuvering abilities of the pursuing (1) and the evading (2) robots participating in the process 
are given by generalized dynamic systems $P_1$,$P_2$ in $X$. 
The generalized dynamic system $P_l$,$l=1,2$ is defined by means of the family of multivalued mappings of the space X into itself, which is 
denoted by $P_l(x^l,t)$, $l=1,2$ and called the attainability function of robot $l$. 
Intuitively, $P_l(x^l,t)$ is the set of points of the space $X$, which the robot $l$ can reach from the $x^l$  in time $t \ge 0$. 
The function $P_l(x^l,t)$ is supposed to meet the following acsioms:

\begin{enumerate}[I.]
	\item $P_l(x^l, t)$ is defined for all $x^l \in X, t \ge 0$ and is supposed to be a nonempty compact set of the space X.
	\item The initial condition: $P_l(x^l, 0)=x^l$ for all $x^l \in X$ is supposed to be valid.
	\item The semigroup property: for all the values $t_1 \le t_2, x_0^l \in X$,
	\[
		P_l(x_0^l,t_2)=\bigcup_{x_1^l \in P_l(x_0^l,t_1)} P_l(x_1^l, t_2 - t_1)
	\]
	\item The function $P_l(x^l, t)$ is supposed to be jointly continuous in the Hausdorff metric.
\end{enumerate}

Let $\Sigma$ be the set of finite partitions $\sigma$ of interval $[0,T]$, $T<\infty$:
$\sigma=0 \le t_1 \le t_2 \le ... \le t_{N_\sigma}$. At every moment $t \in [0,T]$ of the process $\Gamma_i(x_0^1,x_0^2,T), i = 1,2$ each robot is informed of 
the positions of both robots - the points $x^l(t), l=1,2$, and their maneuvering possibilities, defined by the functions $P_l,l=1,2$; 
the duration $T<\infty$ of the process is known as well.
Now we  shall define the strategies of the robots in the process $\Gamma_i(x_0^1,x_0^2,T), i = 1,2$. The strategy $\varphi_l$ if the robot $l$ in the conflict process $\Gamma_i(x_0^1,x_0^2,T)$
it the pair $(\sigma_{\varphi_l},K_\sigma^l)$, where $\sigma_{\varphi_l} \in \Sigma_T$ and $K_\sigma^l$ is a mapping, transforming the pair
\[
	\hat{x}_{t_k}^1 \in \hat{P}_1(x_0^1,t_k), \hat{x}_{t_k}^2 \in \hat{P}_2(x_0^2,t_k), t_k \in \sigma_{\varphi_l}=\sigma_l
\]
into trajectory
\[
	\hat{x}_{t_{k+1}-t_k}^l \in \hat{P}_l(\hat{x}_{t_k}^l(t_k), t_{k+1} - t_k).
\]

The pair $(\varphi_1,\varphi_2)=((\sigma_{\varphi_1},K_{\sigma_1}^1),(\sigma_{\varphi_2},K_{\sigma_2}^2)) \in \Phi_1 \times \Phi_2$ 
is called a situation in the process 
$\Gamma_i(x_0^1,x_0^2,T)$. Then the payoff functions $H_l$ are defined on the set $\Phi_1 \times \Phi_2$.
Let a continuous function $H : X \times X \rightarrow R_1$ be defined on the product $X \times X$

In the process $\Gamma_1(x_0^1,x_0^2,T)$ by every strategy pair $(\phi_1,\phi_2)=\phi$ value
\[
	H_1(\phi)=H(\chi(\phi)(T))=\bar{H}_1(\chi(\phi)),
\]
which is called the payoff of the evading robot, is calculated. In the process $\Gamma_2(x_0^1,x_0^2,T)$ the value
\[
	H_2(\phi)=\min_{t \in [0,T]} H(\chi(\phi)(T))=\bar{H}_2(\chi(\phi))
\]
- the payoff of the pursuing robot - is calculated by every pair $(\phi_1,\phi_2)=\phi$.

In both processes robot 2, choosing the strategy $\Phi_2$, tries to	maximize its
payoff function; the aim of robot 1 is contrary.

Let us remind that if $G : \Phi_1 \times \Phi_2 \rightarrow R_1$, then the pair $(\phi_1,\phi_2)$ is called
a sddle point of the process $\Gamma_G$ ($\varepsilon$-saddle point of the process $\Gamma_G$), if for all 
$\varphi_1' \in \Phi_1$, $\varphi_2' \in \Phi_2$ the inequalities
\begin{eqnarray*}
	G(\phi_1,\varphi_2') \le&G(\phi_1,\phi_2)& \le G(\varphi_1,\phi_2') \nonumber \\
	(G(\phi_1,\varphi_2') - \varepsilon \le&G(\phi_1,\phi_2)& \le G(\varphi_1,\phi_2') + \varepsilon) \nonumber
\end{eqnarray*}
are true. We shall call the process $\Gamma_1(x_0^1,x_0^2,T)=\Gamma_{H_1} = \langle \Phi_1,\Phi_2,H_1 \rangle$
as a process with a terminal payoff, and the process $\Gamma_2(x_0^1,x_0^2,T)=\Gamma_{H_2} = \langle \Phi_1,\Phi_2,H_2 \rangle$
as a conflict process of evasion with a prescribed duration.

In the strategies considered above, the partition of the interval $[0,T]$ 
of the conflict process was chosen by the robot before the process beginning. 
However in some cases it might be convenient to remove such restriction and 
let the the robot choose the point $t_{k+1}$ of the partition $\sigma$ at the moment $t_k$, 
supposing that the resulting partition $\Sigma$ belongs to the set $\Sigma_T$ of finite partitions of $[0, T]$. 
In future we shall call them piecewise-programmed strategies with a non-prescribed in advance partition of the conflict process interval,
and the strategies of the first type - just piecewise-programmed strategies.
Now we shall consider for the process $\Gamma_i(x_0^1,x_0^2,T)$ some auxiliary processes 
$\underline{\Gamma}_i^\sigma(x_0^1,x_0^2,T),\overline{\Gamma}_i^\sigma(x_0^1,x_0^2,T)$ which we shall call lower and upper approximate 
auxiliary processes for the processes $\Gamma_i^\sigma(x_0^1,x_0^2,T)$. Here $\sigma \in \Sigma_T$. Let us suppose that 
the partitions $\sigma = \sigma_n$ are binary, $t_{k+1}-t_k = T/2^n, k = 0,2^n - 1$ and denote 
the set of such partitions by $\Sigma_T^2$. Dynamics of robots 1 an 2 maneuvering abilities in all the auxiliary processes 
$\underline{\Gamma}_i^\sigma(x_0^1,x_0^2,T),\overline{\Gamma}_i^\sigma(x_0^1,x_0^2,T)$ are the same as in the basic processes 
and are defined by means of the generalized dynamic system.

Let the partition $\sigma_n \in \Sigma_T^2$ be fixed now. The process 
$\underline{\Gamma}_i^\sigma(x_0^1,x_0^2,T)$ runs as follows at
the moment $t_0=0$ robot 2, being informed of the initial positions of both
robots $x_0^1$, $x_0^2$, chooses the trajectory $\hat{x}^2[x_0^2, t_1]$, 
$t_1=\sigma_n= T/2^n$. Knowing the initial positions of robots $1$ and $2$
and the trajectory $\hat{x}^2[x_0^2, t1]$ chosen by the robot $2$, robot $1$ 
chooses the trajectory $\hat{x}^1[x_0^1, t_1]$. On the second step, at the 
moment $t_1$, knowing the trajectories $\hat{x}^l[x_0^l, t_1], l=1,2$,
robot 2 chooses the trajectory $\hat{x}^2[x_1^2, \sigma_n]$ for the next interval 
of time $[t_1,t_2]$; and knowing the trajectories $\hat{x}^2[x_1^2,  \sigma_n]$, 
$\hat{x}^1[x_0^1, t_1]$ robot 1 chooses the trajectory
$\hat{x}^1[x_1^1,  \sigma_n]$.

In the same way process continues up to the moment $T$ where it ends. 
As a result, in the process $\underline{\Gamma}_i^\sigma(x_0^1,x_0^2,T)$ 
robot 2 receives from the robot 1 a payoff, equal to the 
$\overline{H}_i(\hat{x}_T^1,\hat{x}_T^2), i=1,2$. 
Here $(\hat{x}_T^1,\hat{x}_T^2)$
is the trajectory, realized in the process 
$\underline{\Gamma}_i^\sigma(\cdot)$.

The process $\overline{\Gamma}_i^\sigma(\cdot)$ proceeds in a dual manner.
At the initial moment $t_0=0$, knowing the initial positions $x_0^1$, $x_0^2$
of the robots, robot 1 chooses the trajectory $\hat{x}^1[x_0^1, t_1]$. Knowing
the trajectory $\hat{x}^1[x_0^1, t_1]$ robot 2 chooses its trajectory 
$\hat{x}^2[x_0^2, t_1]$. The process is repeated in the same way at the 
following steps $2,3,\dotsc,N_\sigma$. At the $N_\sigma$-th step, 
the process ends after which, in the process 
$\overline{\Gamma}_i^\sigma(x_0^1,x_0^2,T)$
the robot 2 receives from the robot 1 the payoff equal to the
$\overline{H}_i(\hat{x}_T^1,\hat{x}_T^2), i=1,2$. 
Here $(\hat{x}_T^1,\hat{x}_T^2)$
is the trajectory, realized in the process 
$\overline{\Gamma}_i^\sigma(\cdot)$.

For proving the existence theorems for equilibrium points in the process
$\Gamma_i(\cdot)$, there will be a need for an auxiliary process 
$\hat{\Gamma}_i^\sigma(\cdot)$, which is truncation of the process 
$\overline{\Gamma}_i^\sigma(\cdot)$ at the last step. This process differs
from the process $\overline{\Gamma}_i^\sigma(\cdot)$ only in the fact
that in this process at the last $N_\sigma$-th step robot 2 does not make 
a choice of the trajectory $\hat{x}^2[x_{N_\sigma-1},\sigma_n]$.

Now let us formulate several auxiliary statements: 
\begin{enumerate}[1.]
	\item In the processes $\overline{\Gamma}_i^\sigma(x_0^1,x_0^2,T)$,
	$\underline{\Gamma}_i^\sigma(\cdot)$, $\hat{\Gamma}_i^\sigma(\cdot)$
	saddle points in pure strategies exist, the functions of value 
	$Val(\overline{\Gamma}_i^\sigma(\cdot) )$,
	$Val(\underline{\Gamma}_i^\sigma(\cdot))$,
	$Val(\hat{\Gamma}_i^\sigma(\cdot))$ are separately continuous by $x_0^1$,
	$x_0^2$. For any partition $\sigma \in \Sigma_T^2$ the inequality
	\[
		Val(\overline{\Gamma}_i^\sigma(x_0^1,x_0^2,T) ) \ge 
		Val(\underline{\Gamma}_i^\sigma(x_0^1,x_0^2,T))
	\]
	is satisfied.
	\item For any pairs of partitions $\sigma,\sigma' \in \Sigma_T$,
	such that $\sigma'$ is a refinement of $\sigma$:
	\begin{eqnarray*}
		Val(\overline{\Gamma}_i^\sigma(x_0^1,x_0^2,T) ) \ge 
		Val(\overline{\Gamma}_i^{\sigma'}(x_0^1,x_0^2,T)), \\
		Val(\underline{\Gamma}_i^\sigma(x_0^1,x_0^2,T) ) \le 
		Val(\underline{\Gamma}_i^{\sigma'}(x_0^1,x_0^2,T)).
	\end{eqnarray*}
	\item For any sequence $\{\sigma_n\}_{n=1}^\infty$ of partitions
	of the interval $[0,T]$, $\sigma_n \in \Sigma_T^2$, such that 
	$|\sigma_n| \xrightarrow[n \to \infty]{}0$:
	\[
		\lim_{n \to \infty}Val(\overline{\Gamma}_i^{\sigma_n}(x_0^1,x_0^2,T) ) = 
		\lim_{n \to \infty}Val(\underline{\Gamma}_i^{\sigma_n}(x_0^1,x_0^2,T) )
	\]
	\item For any sequences $\{\sigma_n\}_{n=1}^\infty, 
	\{\sigma'_n\}_{n=1}^\infty;\sigma_n, \sigma'_n \in \Sigma_T$, such that
	$|\sigma_n| \xrightarrow[n \to \infty]{}0$,
	$|\sigma'_n| \xrightarrow[n \to \infty]{}0$ the equality
	\[
		\lim_{n \to \infty}Val(\overline{\Gamma}_i^{\sigma_n}(x_0^1,x_0^2,T) ) = 
		\lim_{n \to \infty}Val(\overline{\Gamma}_i^{\sigma'_n}(x_0^1,x_0^2,T) )
	\]
	is true.
\end{enumerate}

These statements allow us to formulate the following theorem:

\begin{theorem}
	For any $x_0^1,x_0^2 \in X$, $T < \infty$ and $\varepsilon > 0$ in the process $\Gamma_1(x_0^1,x_0^2,T)$ there exists a pair of 
	$\varepsilon$-guaranteeing strategies, and besides
	\[
		Val(\overline{\Gamma}_1(x_0^1,x_0^2,T)) = 
		\lim_{n \to \infty}Val(\overline{\Gamma}_1^{\sigma_n}(x_0^1,x_0^2,T)),
	\]
	where $\{\sigma_n\}_{n=1}^\infty$ - is any refinery sequence of partitions
	$[0,T]$, $|\sigma_n| \xrightarrow[n \to \infty]{}0$
\end{theorem}
\begin{proof}
	To prove that, let us fix the number $\varepsilon > 0$ and show, that 
	such strategies $\varphi_1^\varepsilon \in \Phi_1, \varphi_2^\varepsilon \in \Phi_2$ can be found, that for any strategies 
	$\varphi_1 \in \Phi_1, \varphi_2 \in \Phi_2$ the inequalities 
	\[
		H_1(\varphi_1^\varepsilon, \varphi_2) - \varepsilon
		\le
		H_1(\varphi_1^\varepsilon, \varphi_2^\varepsilon)
		\le
		H_1(\varphi_1, \varphi_2^\varepsilon) - \varepsilon
	\]
	are true.

	According to the statements 3 and 4 there can be found such partitions
	$\sigma_{1,\varepsilon}, \sigma_{2,\varepsilon} \in \Sigma_T$ that
	\begin{eqnarray*}
		Val(\overline{\Gamma}_1^{\sigma_{1,\varepsilon}}(\cdot))
		-
		\lim_{n \to \infty}Val(\overline{\Gamma}_1^{\sigma_n}(\cdot))
		< \varepsilon, \\
		\lim_{n \to \infty}Val(\underline{\Gamma}_1^{\sigma_n}(\cdot))
		-
		Val(\underline{\Gamma}_1^{\sigma_{2,\varepsilon}}(\cdot))
		< \varepsilon.
	\end{eqnarray*}
	Let $\varphi_l^\varepsilon = (\sigma_{l,\varepsilon}, K_{\sigma_{l,\varepsilon}}^l), l=1,2$ to be optimal strategies for robots 1 and 2 in
	the processes 
	$\overline{\Gamma}_1^{\sigma_{1,\varepsilon}}(x_0^1,x_0^2,T)$
	and
	$\underline{\Gamma}_1^{\sigma_{2,\varepsilon}}(x_0^1,x_0^2,T)$

	It follows from the definition of a strategy that the pair 
	$(\varphi_1^\varepsilon,\varphi_2^\varepsilon)$ is compatible, i.e. the
	only one process trajectory can be built for it in an only one way.
	Here, because of the choice of the mapping $K_{\sigma_{l,\varepsilon}}^l, l=1,2$,
	by means of the strategy $\varphi_1^\varepsilon$ robot 1 guarantees for
	itself a payoff not smaller than 
	\[
		\lim_{n \to \infty}Val(\underline{\Gamma}_1^{\sigma_n}(x_0^1,x_0^2,T))
		+ \varepsilon,
	\]
	and by means of the strategy $\varphi_2^\varepsilon$ robot 2 guarantees
	itself a payoff not smaller than 
	\[
		\lim_{n \to \infty}Val(\underline{\Gamma}_1^{\sigma_n}(x_0^1,x_0^2,T))
		- \varepsilon.
	\]
	Consequently, $\varphi_l^\varepsilon, l=1,2$ are the pair of the 
	$\varepsilon$-guaranteeing strategies strategies for both robots 
	and the function
	\[
		Val(\overline{\Gamma}_1(x_0^1,x_0^2,T)) = 
		\lim_{n \to \infty}Val(\overline{\Gamma}_1^{\sigma_n}(x_0^1,x_0^2,T)),
	\]
	is the function of value of the process $\overline{\Gamma}_1(x_0^1,x_0^2,T)$ \qed
\end{proof}

By analogy, for the process $\overline{\Gamma}_2(x_0^1,x_0^2,T)$ the following
theorem can be proved:
\begin{theorem}
	For any $x_0^1,x_0^2 \in X$, $T < \infty$ and $\varepsilon > 0$ in the process $\Gamma_2(x_0^1,x_0^2,T)$ there exists a pair of 
	$\varepsilon$-guaranteeing strategies. Here
	\[
		Val(\overline{\Gamma}_2(x_0^1,x_0^2,T)) = 
		\lim_{n \to \infty}Val(\overline{\Gamma}_2^{\sigma_n}(x_0^1,x_0^2,T)),
	\]
\end{theorem}

Now let us consider the processes with a time of capture payoff 
$\Gamma(x_0^1,x_0^2)$ 
running on the interval $[0,\infty)$, 
which makes them differ from the processes with prescribed duration. 
The process $\Gamma(x_0^1,x_0^2)$  runs in the full local compact 
metric space $X$, the maneuvering abilities of robots 1 and 2 are defined by 
means of generalized dynamic system.

Information states of both robots in the process $\Gamma(x_0^1,x_0^2)$
are the same as in the processes $\Gamma_i(x_0^1,x_0^2,T), i=1,2$. 
The strategy $\varphi_l$ of robot $l$ in process $\Gamma(x_0^1,x_0^2)$
is the pair $(\sigma_l, K_{\sigma_l})$, where $\sigma_l=\{t_0=0<t_1<\cdots<t_k<\cdots\}$ is a partition of the semiline $[0,\infty)$, containing no limit
 points, and $K_{\sigma_l}$ is a mapping, transforming the information 
 state of robot $l$ at the moment $t_k \in \sigma_l$ into the trajectory
 $\hat{x}^l \in \hat{P}_l(x_l^k, t_{k+1}-t_k)$. We shall denote the strategy set of robot $l$ in the process $\Gamma(x_0^1,x_0^2)$ by $\Phi_l$.

 As it was in the case of prescribed duration for the processes with a time of 
 capture a payoff for each pair $(\varphi_1,\varphi_2)$ has one and only one 
 corresponding pair 
 of the trajectories of robots 1 and 2, which are defined on the ray 
 $[0,\infty]$ and will be denoted by
\[
	(\hat{x}_\infty^1,\hat{x}_\infty^2)=\chi(\varphi_1,\varphi_2)
\]

Now let $M$ be a non-empty closed set of $X \times X$. Let us define the 
payoff function for the pair $(\varphi_1, \varphi_2)$ in this way:
\begin{eqnarray*}
	&H_\alpha(\varphi_1, \varphi_2)=H(\chi(\varphi_1,\varphi_2))=
	\min_{t \in [0.\infty]}\{t\ | \ \chi(\varphi_1,\varphi_2) \in M_\alpha\}& 
	\\
	&(M_\alpha = \{z \in X \times X \ | \ \rho(z,M) \le \alpha\})&
\end{eqnarray*}

If $H_\alpha(\varphi_1, \varphi_2)=\infty$, then for the pair
$(\varphi_1, \varphi_2)$ the process $\Gamma(x_0^1,x_0^2)$ can not end
within a finite time. By choosing the strategy $\varphi_1$ robot 1 tries to 
minimize its payoff function, the aim of robot 2 is opposite.

We shall call the strategy $\varphi_1 \in \Phi_1$ successful, if for any
strategy $\varphi_2$  in the situation $(\varphi_1, \varphi_2)$ the process $\Gamma(x_0^1,x_0^2)$ ends within a finite time.

Thus, having defined the strategy sets of robots 1 and 2 in the process as 
well as the payoff function on the products of these sets, we have defined 
the process $\Gamma(x_0^1,x_0^2)$ in the normal form.

Let us consider the following theorem:
\begin{theorem}
	If in the process $\Gamma(x_0^1,x_0^2)$ robot 1 has a successful strategy
	for every $\alpha > 0$, then in this process for every $\varepsilon > 0$ 
	there exists a pair of $\varepsilon$-guaranteeing strategies.
\end{theorem}
\begin{proof}
	To prove it, let us choose an arbitrary $\delta > 0$. Let us set
	\[
		\underline{V}_{\alpha/2} = 
		\sup_{\{\varphi_2\}}\inf_{\{\varphi_2\}}H_{\alpha/2}(\varphi_1, \varphi_2).
	\]
	It follows from this, that for every strategy $\varphi_2$ there can be
	found such $\varphi_1$, that
	\[
		H_{\alpha/2}(\varphi_1, \varphi_2) \le
		 \underline{V}_{\alpha/2} + \delta
	\]
	Let us consider the process 
	$\Gamma_2(x_0^1,x_0^2,\underline{V}_{\alpha/2} + \delta)$ with the
	payoff function 
	\[
		H_2(\varphi_1, \varphi_2) = 
		\min_{t \in [0, \underline{V}_{\alpha/2} + \delta]}\rho(\chi(\varphi_1, \varphi_2)(t), M_{\alpha/2})
	\]
	From the previous theorem there exist an equilibrium point for every 
	$\varepsilon > 0$ in this process and besides
	\[	
		Val(\Gamma_2(x_0^1,x_0^2,\underline{V}_{\alpha/2} + \delta))=0
	\]
	It means that for every $\varepsilon > 0$ such a strategy 
	$\varphi_{1,\varepsilon}$ can be fond, that for every strategy 
	$\varphi_2$ robot 1 is guaranteed to approach the set $M_{\alpha/2}$
	at distance $\varepsilon$ within the time 
	$\underline{V}_{\alpha/2} + \delta$ for every $\delta > 0$ and, 
	consequently, for every $\alpha > 0$ the set $M_\alpha$ within the time
	$\underline{V}_{\alpha/2}$.

	In the same way, for every $\alpha > 0$ and $\varepsilon > 0$ there exists 
	such a strategy $\varphi_{2,\varepsilon}$, which guarantees the robot 2
	the relation $\varphi_2(t) \notin {M_{alpha/2}}$ within the time 
	$\underline{V}_{\alpha/2} - \varepsilon$. \qed
\end{proof}

Let us suppose that $M=\{(x,x) \in X \times X\}$, i.e. it is a diagonal in 
$X \times X$.
We shall define
\[
	T^*(x_0^1,x_0^2) = 
	\min_{t\in[0,\infty]}\{ t \ | \ P_2(x_0^2,t) \subset P_1(x_0^1,t) \}
\]
and suppose that $T^*<\infty$.

Let us formulate the following theorem:
\begin{theorem}
	If in the process $\underline{\Gamma}^\sigma(z_0,T^*)$ with the
	payoff function 
	\[
		H(\varphi_1,\varphi_2)=
		\min_{t\in[0,T^*]}\rho(\chi(\varphi_1,\varphi_2)(t), M), 
		(z_0=(x_0^1,x_0^2))
	\]
	for every $\sigma\in\Sigma_T$ robot 1 has a strategy, guaranteeing
	a zero payoff for it, then in the process $\Gamma(z_0)$ there exists a 
	pair of $\varepsilon$-guaranteeing strategies for every 
	$\varepsilon > 0, \alpha >0$.
\end{theorem}
\begin{proof}
	In fact, according to the last two theorems 
	\[
		Val(\Gamma(z_0,T^*)) = 
		\lim_{n \to \infty}Val(\underline{\Gamma}^{\sigma_n}(z_0,T^*))
	\]
	and from the condition of theorem we have that 
	$Val(\Gamma(z_0,T^*)) = 0$. Consequently,
	in the process $\Gamma(z_0)$ there exists an equilibrium point
	for every $\varepsilon > 0,\alpha >0$. \qed
\end{proof}

Thus there are $\varepsilon$-optimal strategies for the robots in dynamic 
conflict processes and an algorithm for their numerical calculation is offered
in \cite{b1},\cite{b2}.

In that algorithms the continuous conflict process $\Gamma_1(\cdot)$ and 
$\Gamma_1(\cdot)$ are approximated by discrete multi-step processes with the 
informational discrimination of the pursuing robot $P_1$ or the evading robot 
$P_2$. For those multi-step processes, numerical algorithms of the robot 
optimal strategies calculation are constructed on the dynamic programing 
approach basis, and also numerical bound of the quality functional optimal 
value are elicited.

\section{Example Implementation}
Let us consider the following simple implementation of the ideas, described above. We will generate a rough approximation of the upper game under the following limitations:
\begin{enumerate}[1.]
\item The dynamics of each robot is defined by the dynamic system $\dot{x}_i=u_i, | u_i | \le const < \infty, i = 1,2$
\item The positions of the obstacles are known at every moment of time
\item Both robots are using the same uniform time partition with a time step $\Delta_t$
\item All obstacles are circle-shaped
\item The payoff function is defined as the distance between two robots at the final moment $T$
\end{enumerate}

More over, the decision tree is built as follows: the boundary of the attainability set (which is, obviously, a circle of radius $\Delta_t  u_i$) is uniformly subdivided, i.e. the subdivision is parametrized in terms of angular steps. Such approximation leads to solution instability when the $u_i$ value is high, however, it provides feasible results for small enough $u_i \le 1$ and is easy to implement. Thus, at each step $[2\pi / \Delta_\alpha] + 1$ tree branches are generated, where $\Delta_\alpha$ is an angular step. 

The algorithm runs as described below:
\begin{enumerate}[1.]
\item The decision tree is generated for each robot, as was mentioned before.
\item For each obstacle we check, whether the new robot position intersects a capsule, surrounding the obstacle positions at times $t_{i-1}$ and $t_i$. Whenever the intersection occurs - we remove the specific tree branch (i.e. we are removing a point from a attainability set boundary subdivision)
\item We calculate the payoff function for each point of the boundary of the attainability set subdivisions of both robots.
\item Second robot finds the best tree branch (i.e. with a maximum payoff) for each branch of the decision tree of the first robot
\item From the set of maxims, generated on the previous step robot 1 finds the branch, guaranteeing the minimal payoff
\item Positions, defined by the best tree branches are used at the next iteration of the algorithm
 
\end{enumerate} 

The algorithm runs for a predefined number of steps, calculated from the simulation time limit $T$ and the time subdivision interval $\Delta_t$.

The following trajectories are generated by the algorithm described:
\\
\begin{figure}[h]
	\centering
	\begin{subfigure}[t]{0.4\textwidth}
		\includegraphics[width=0.9\textwidth]{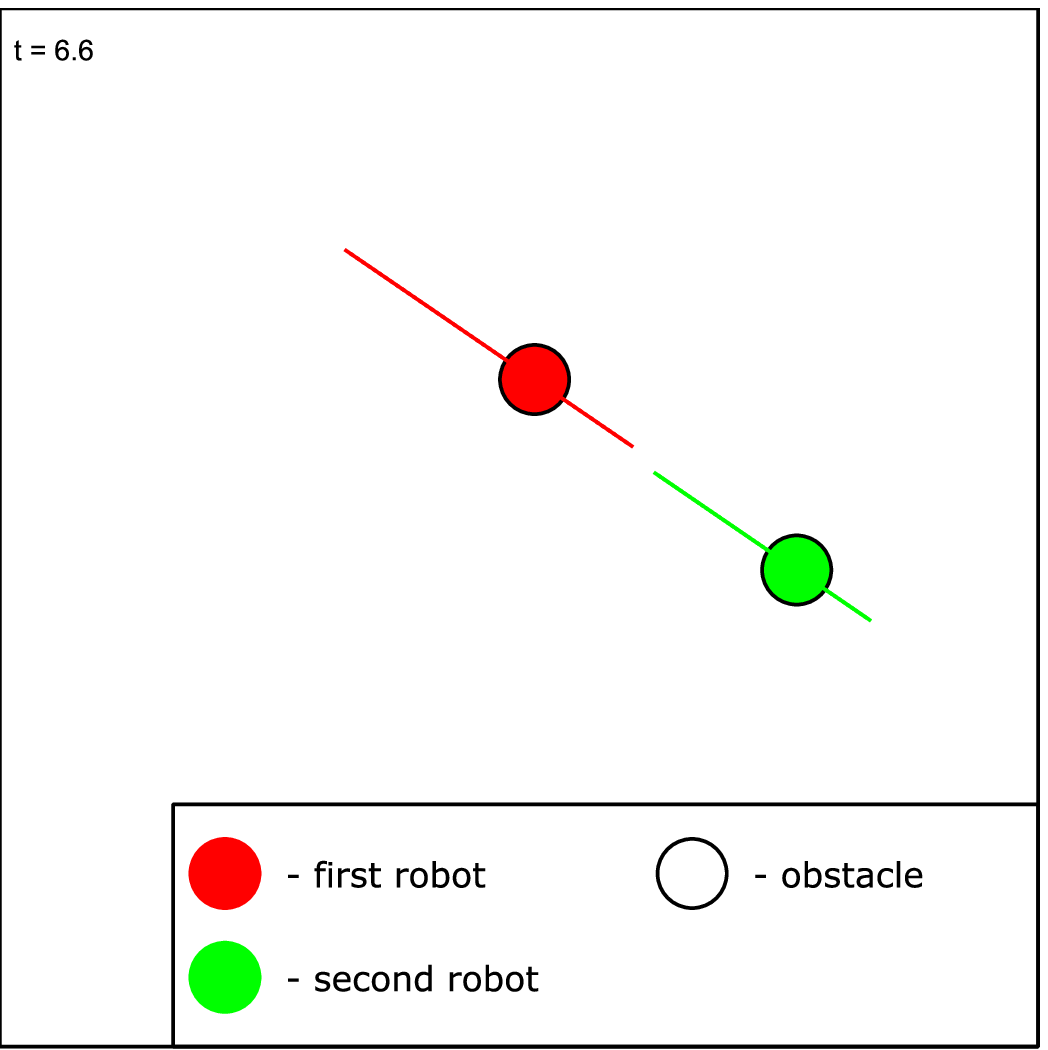}
		\caption{No obstacles}
	\end{subfigure}
	\begin{subfigure}[t]{0.4\textwidth}
		\includegraphics[width=0.9\textwidth]{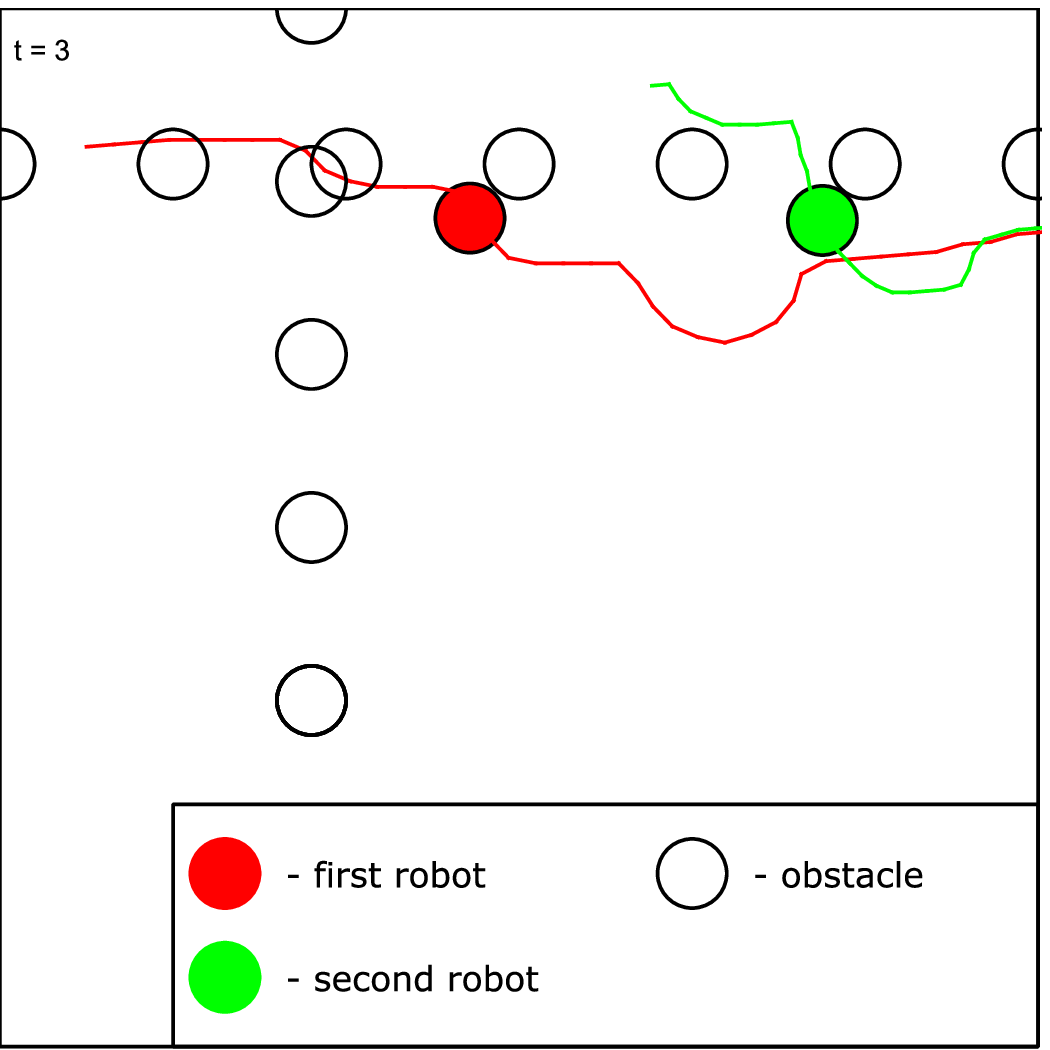}
		\caption{Simple trajectories}
	\end{subfigure}
\begin{subfigure}[t]{0.4\textwidth}
		\includegraphics[width=0.9\textwidth]{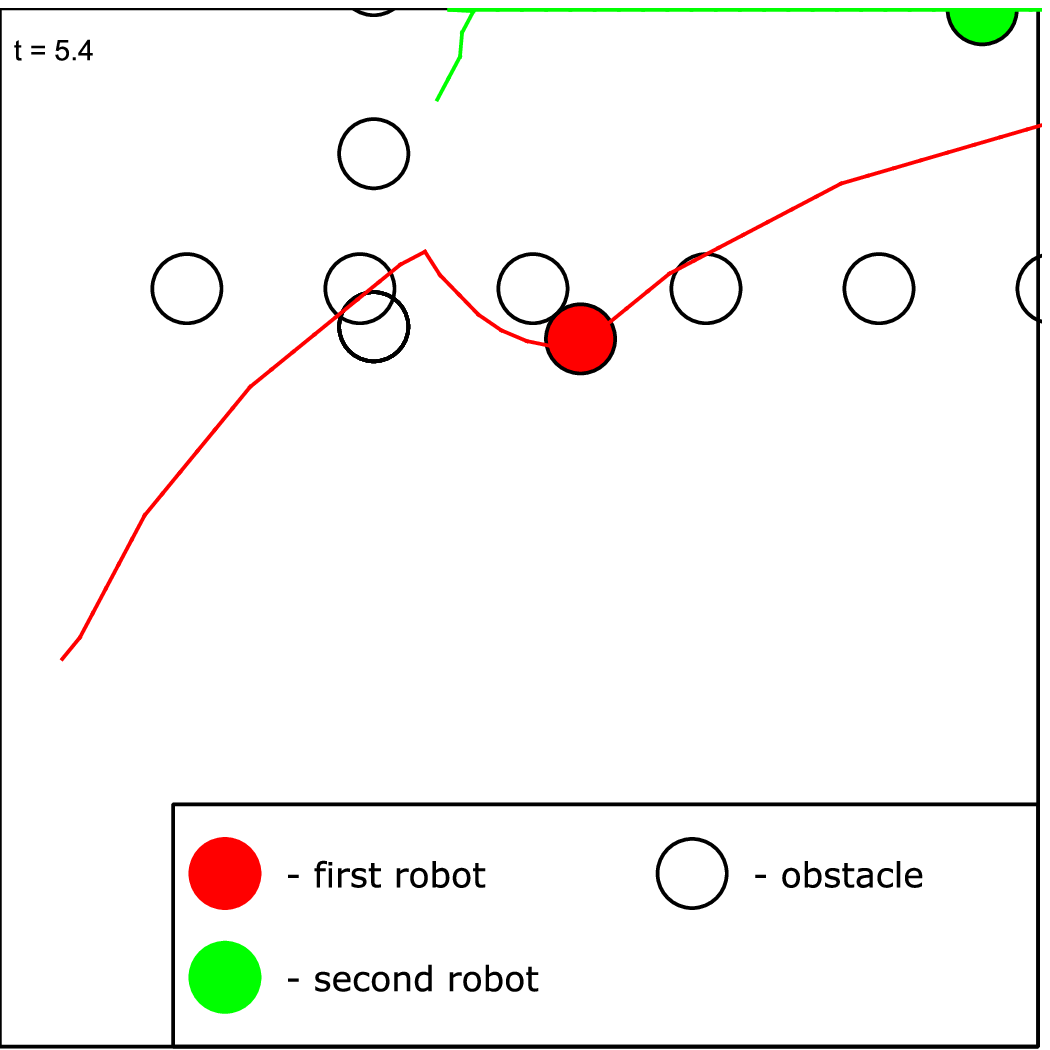}
		\caption{Diagonal trajectories }
	\end{subfigure}
	\begin{subfigure}[t]{0.4\textwidth}
		\includegraphics[width=0.9\textwidth]{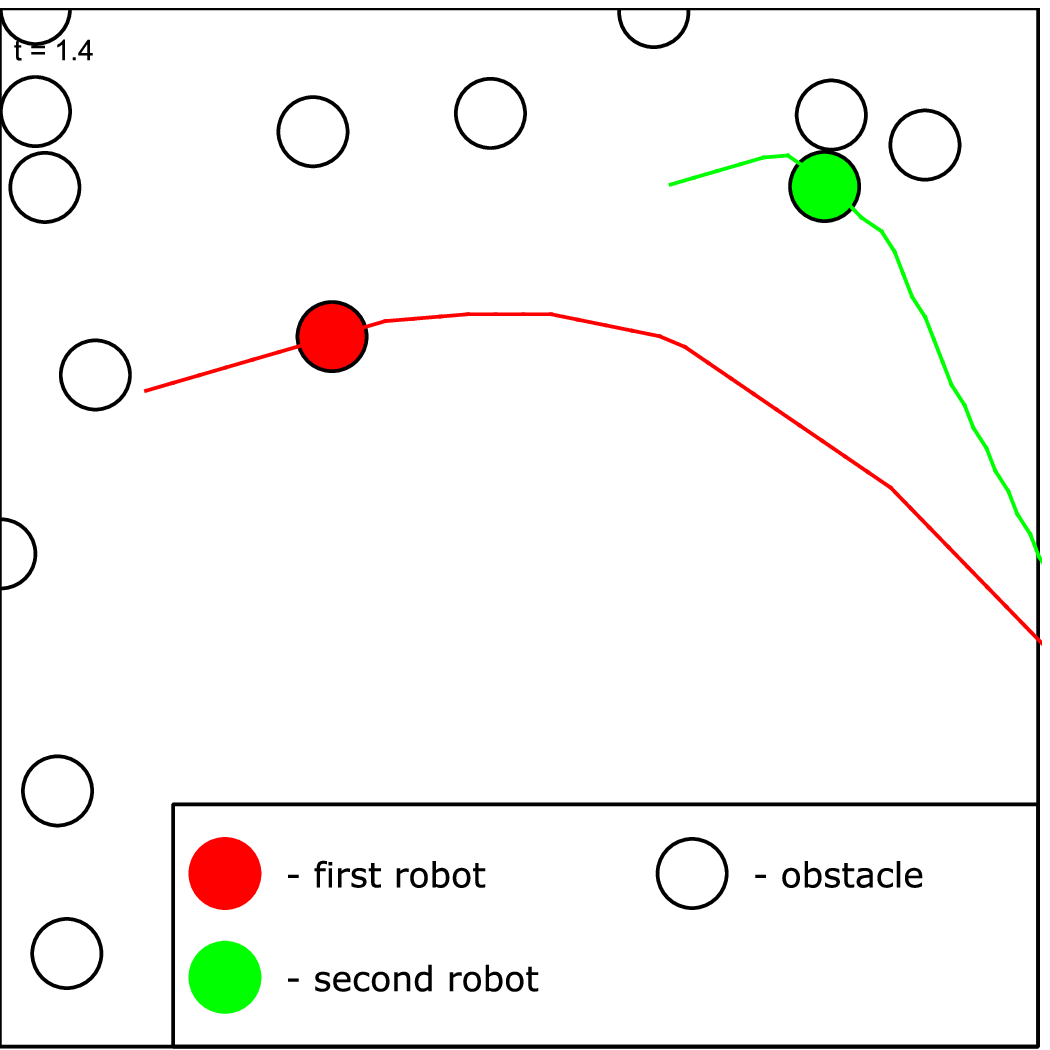}
		\caption{Haotic trajectories}
	\end{subfigure}
\end{figure}

On figure (a) we can see, that the algorithm is perfectly valid for the situations, when no obstacles are present. Next, we add a couple of obstacles to the scene, moving with the constant speed either vertically or horizontally (fig. b) or diagonally (fig. c). 
To make the situation more complicated - we select the random trajectories for the obstacles, which results in trajectories shown on figure (d). In all the cases $T=10$, $\Delta_t = 0.2$, $\Delta_\alpha=0.2$ and the dynamics of the robots are defined by the following
ODE system:
\begin{eqnarray*}
	&&\dot{x}_1 = u_1, | u_1 | <= 10 \\
	&&\dot{x}_2 = u_2, | u_2 | <= 8
\end{eqnarray*}

As it can be seen, algorithm gives expected results for both robots and according to the theorems - provides an approximation of the optimal trajectories, thus solving the problem stated in this paper.

The algorithm worst-case complexity is $ O( n_t n_a n_o ) $, where $n_t$, $n_a$, $n_o$ are respectively the number of time steps, angular step and obstacles. It can be improved to $ O( n_t n_a)$ average-case complexity
by using the spatial hashing to improve the collision detection with the obstacles.
\newline
\section{Acknowledgements}
The work is partly supported by work RFBR No. 18-01-00796.

\end{document}